\documentclass[reqno]{amsart}

\usepackage[lite]{amsrefs}
\usepackage{amscd}
\usepackage{pb-diagram}
\usepackage{color}
\usepackage{amssymb}
 \usepackage{amsmath}
         \usepackage{extarrows}
          \usepackage{amsfonts}
          \usepackage[english]{babel}
          \usepackage[utf8]{inputenc}
          \usepackage{enumerate}
          \usepackage{graphicx}
          \usepackage[colorlinks, linkcolor=blue, citecolor=red, bookmarks=true, bookmarksnumbered=true]{hyperref}
\usepackage[all,cmtip]{xy}

\numberwithin{equation}{section}

\theoremstyle{plain} 
\newtheorem{thm}{Theorem}[section]
\newtheorem{cor}[thm]{Corollary}
\newtheorem{lem}[thm]{Lemma}
\newtheorem{prop}[thm]{Proposition}

\theoremstyle{definition}
\newtheorem{defn}[thm]{Definition}

\theoremstyle{remark}
\newtheorem{rem}[thm]{Remark}

\def\fin{$\diamondsuit$\hskip-0.26truecm \raise0.4ex\hbox{$\checkmark$}}

\begin{document}

\makeatletter
\numberwithin{equation}{section}
\renewcommand{\sectionmark}[1]{\markright{
	\thesection. #1}{}}

\setcounter{section}{0}
 \def\@makechapterhead#1{%
  \vspace*{25\p@}%
  {\parindent \z@ \raggedright \normalfont
    \ifnum \c@secnumdepth >\m@ne
      \if@mainmatter
        \centering\LARGE\bfseries \@chapapp\space \thechapter
        \par\nobreak
        \vskip 10\p@
      \fi
    \fi
    \interlinepenalty\@M
    \huge \bfseries #1\par\nobreak
    \vskip 30\p@
  }}


  \def\section{\@startsection {section}{1}{\z@}
                                   {-3.5ex \@plus -1ex \@minus -.2ex}
                                   {2.3ex \@plus.2ex}
                                   {\normalfont\large\bfseries\centering\S\;\;}}
  \def\subsection{\@startsection {subsection}{1}{\z@}
                                   {-3.5ex \@plus -1ex \@minus -.2ex}
                                   {2.3ex \@plus.2ex}
                                   {\normalfont\large\bfseries\centering\ \;\;}}
\makeatother

\title[AUTOMORPHISMS FOR $\sigma-$PBW EXTENSIÓN AND QUANTUM POLINOMIAL RINGS]{AUTOMORPHISMS FOR SKEW $PBW$ EXTENSIONS AND SKEW QUANTUM POLYNOMIAL RINGS}

\author{C\'esar Fernando Venegas Ram\'irez\\
José Oswaldo Lezama Serrano}
\address{Seminar of Constructive algebra $SAC^2$\\ Universidad Nacional de Colombia\\
Bogot\'a, Colombia}
\email{cfvenegasr@unal.edu.co}
\date{\today}

 \begin{abstract}

In this work we study the automorphisms of skew $PBW$ extensions and skew quantum polynomials. We use Artamonov's works as reference for getting the principal results about automorphisms for  generic skew $PBW$ extensions  and generic skew quantum polynomials. In general, if we have an endomorphism on a generic skew $PBW$ extension and there are some $x_i,x_j,x_u$  such that the endomorphism is not zero on this elements and the principal coefficients are invertible, then   endomorphism act over $x_i$ as $a_ix_i$ for some $a_i$ in the ring of coefficients. Of course, this is valid for quantum polynomial rings, with $r=0$,  as such Artamonov shows in his work. We use this result for giving some  more general results  for skew $PBW$ extensions, using filtred-graded  techniques. Finally, we use localization for characterize some class the endomorphisms and automorphisms for skew $PBW$ extensions and skew quantum polynomials over Ore domains.

 \end{abstract}



\maketitle

 \tableofcontents

{\color{red}
 \section{Introduction}\label{sec:intro}}
 Let $D$ be a division ring with a fixed set $\alpha_1,\ldots, \alpha_n$, $n\geq 2$, of its automorphisms. We have fix elements $q_{ij} \in D^*$, $i,j=1,\ldots, n$, satisfying  the equalities  
 \begin{eqnarray}
 q_{ii} = q_{ij} q_{ji} = Q_{ijr} Q_ {jri} Q_{rij} = 1, \qquad 
 \alpha_i (\alpha_j (d)) = q_{ij} \alpha_j (\alpha_i (d))q_{ji} ,
 \end{eqnarray}
 where
 
 \begin{eqnarray}
 Q_{ijr} = q_{ij} \alpha_j (q_{ir} ), \text{ and } d \in D.
 \end{eqnarray}
 
 \begin{eqnarray}
 Q = (q_{ij} ) \in Mat_n(D) \text{ and } \alpha = (\alpha_1 , \ldots , \alpha_n ).
 \end{eqnarray}
 
 \begin{defn}
 The entries $q_{ij}$ of the matrix $Q$ form  a system of multiparameters. 
 \end{defn}
 \begin{defn}[Quantum polynomial ring]\label{definitionquamtumring}
 Denote by 
 \begin{eqnarray}
 \Lambda:=D_{Q,\alpha}[x_1^{\pm 1}, x_2^{\pm 2},\ldots, x_r^{\pm 1},x_{r+1}, \ldots, x_n] \label{quantumpolynomial}
 \end{eqnarray}
 the associative ring generated by $D$ and by the elements 
 
 \begin{eqnarray}
 x_1,\ldots, x_n, x_1^{- 1}, \ldots, x_r^{- 1}, 
 \end{eqnarray}
 subject to  the defining relations 
 \begin{eqnarray}
 x_ix_i^{-1}&=& x_i^{-1}x_i=1,\qquad 1\leq i \leq r,\\
 x_id&=&\alpha_i(d)x_i, \qquad d \in D, \quad i=1,\ldots, n \\
 x_ix_j&=& q_{ij}x_jx_i, \qquad i,j=1,\ldots, n.
 \end{eqnarray}
 \end{defn}
 
 We denoted by $End(\Lambda)$ the semigroup of all ring endomorphisms of $\Lambda$ acting identically on $D$. Denote by $Aut(\Lambda)$ the group of all ring automorphisms of $\Lambda$, acting identically on $D$, i.e., the invertible elements of the semigroup $End(\Lambda)$.
 {\color{red}
 \subsection{Automorphisms for the quantum space.}
 }
Alev and Chamarie (see \cite{alev-chamarie}) show some results very interesting on automorphisms of quantum space over fields $k$ of characteristic zero, using a des\-crip\-tion for derivations. Remember that the quantum space is treated as the ring \eqref{quantumpolynomial}, for the case in that $D=k$ is a field, $r=0$ and $\alpha_1=\alpha_2= \cdots =\alpha_n=id_k$. 

\begin{lem}
A matrix $\alpha:=(\alpha_{ij}) \in Mat_n(k)$ define a  linear automorphism of $\Lambda$ if only if $\alpha$ satisfy the following relation for all $i,j,k,l$; $i<j$ and $k \leq l$:
\begin{eqnarray}
\alpha_{ik}\alpha_{jl}(1-q_{ij}q_{lk})=\alpha_{il}\alpha_{jk}(q_{ij}-q_{lk}). \label{ecuacion1alev}
\end{eqnarray}
\end{lem}

\begin{defn}
An automorphism $f$ is linear if it  fixed to the subspace generated by $x_i$. We denote for $Aut_L(\Lambda)$ the set of all linear  automorphisms of $\Lambda$.  
\end{defn}

Denoted by $Der(\Lambda)$ all derivation of the algebra $\Lambda$. Alev and Chamarie tell us that the derivation of the quantum space can be described as
\[ Der(\Lambda)=Derint(\Lambda) \oplus \left( \bigoplus_i Z(\Lambda)D_i \right) \oplus E,\]
where $Derint(\Lambda)$ are the inner derivations, $Z(\Lambda)$ is the center of the ring $\Lambda$, and $E=\bigoplus\limits_i \left( \bigoplus\limits_{v \in \Lambda_i}kD_{iv} \right)$, with $\Lambda_i:=\{ v \in \mathbb{N}^n: v_i=0 \text{ and } \prod\limits_k q_{kj}^{v_k}=q_{ij} \text{ for all } j \neq i \}$ and $D_{iv}(x_j)=\delta_{ij}x^v$.
\begin{thm}
If $E=0$, then $Aut_L(\Lambda)=Aut(\Lambda)$.
\end{thm}

The following three theorems show results for some special cases of quantum spaces (see \cite{alev-chamarie}).

\begin{thm}
For the quantum plane (n=2), we have the following cases:
\begin{enumerate}
\item if $q_{12}\notin \{ 1,-1\}$, then $Aut(\Lambda)$ is isomorphism to $(k^*)^2$, which acts naturally on $kx_1 \oplus kx_2$.
\item If $q_{12}=-1$, then $Aut(\Lambda)$ is isomorphism to the semidirect product between $(k^*)^2$ and $\langle \tau \rangle$, where $\tau$ es the symmetry that exchange $x_1$ and $x_2$.
\end{enumerate}
\end{thm}

\begin{thm}
If $A=\bigotimes\limits_{i=1}^pk_{q_i}[x_i,y_i]$ is a tensor product of $p$ quantum planes $k_{q_i}[x_i,y_i]$, then $Aut(A)$ is isomorphisms at the semidirect product of $(k^*)^{2p}$ y $G$, where $G$ is a subgroup of the group that permute  the $x_i$ and $y_i$.
\end{thm}

\begin{thm}
we suppose that $q_{ij}=q$ for all $i<j$, such that $q$ is not a root of unit. Then, we have the following cases:
\begin{enumerate}
\item If $n=3$, the automorphisms $\sigma$ of $\Lambda$ have the form $\sigma(x_i)=\alpha_ix_i$, $\sigma(x_2)=\alpha_2x_2+\beta x_1x_3$, $\sigma(x_3)=\alpha_3x_3$, with $\alpha_1,\alpha_2,\alpha_3 \in k^*$ and $\beta \in k$; so $Aut(\Lambda)$ is isomorphisms at the semidirect product of $k$ and $(k^*)^3$.
\item If $n\neq 3$, $Aut(\Lambda)$ is isomorphisms to $(k^*)^n$.
\end{enumerate}
\end{thm}
{\color{red}
\subsection{Automorphisms of general quantum polynomial rings.}
}
\begin{defn}
Let $N$ be the subgroup of the multiplicative group $D^*$ of the division ring $D$ generated by the derived subgroup $[D^*, D^*]$ and the set of all elements of the form $z^{-1}\alpha_i(z)$, where $z \in D^*$ and $i=1,\ldots, n$.
\end{defn}

The normal subgroup $N$ always appears when we multiply monomials in the ring $\Lambda$.

\begin{defn}
Whenever the images of all multiparameters $q_{ij}$, $1\leq i <j \leq n$, are independent in the group $D^*/N$, we call $\Lambda$ the a general quantum polynomial rings.
\end{defn}

For this subsection, we assume that A is an general quantum polynomial ring. The following results are presents for Artamonov in \cite{va1997} and \cite{va2001-2}.

\begin{thm}[\cite{va1997}]\label{teorem1.12-2001}
Assume that the automorphisms $\alpha_1,\ldots, \alpha_n$ act identically on $D$. Suppose that $r=0$ and the multiparameters $q_{ij}$, $1\leq i <j \leq n$, $n\geq 3$, are independent in the multiplicative group $D^*$. If $\lambda \in End(\Lambda)$ and all $\lambda(x_1),\ldots, \lambda(x_n)\neq 0$, then $\lambda \in Aut(\Lambda)$ and $Aut(\Lambda)=(D^*)^n$.
\end{thm}

The next theorem is more general that the theorem \ref{teorem1.12-2001}. The proof and the statement were altered from its original version in \cite{va2001-2}.

\begin{thm}[\cite{va2001-2}]\label{theorem2.1-2001}
Suppose that $\lambda \in End(\Lambda)$ and that there exist al least three distinct indices $1\leq i,j,t \leq n$ such that $\lambda(x_i), \lambda(x_j), \lambda(x_t)\neq 0$. Then there exist elements $\lambda_1,\ldots,\lambda_n \in D$ and an integer $\epsilon=\pm 1$ such that $\lambda_1,\ldots, \lambda_r \neq 0$, and 
\begin{eqnarray}
\lambda(x_w)=\lambda_wx_w^\epsilon, \qquad w=1,\ldots, n.
\end{eqnarray}

\end{thm}

\begin{proof}
We consider the lexicographical order for the exponents of the monomials in A. We suppose that $\lambda(x_i)$ and $\lambda(x_j)$ are non-zero, and with leading monomial $\lambda_ix_1^{l_1}\cdots x_n^{l_n}$ and $\lambda_jx_1^{t_1}\cdots x_n^{t_n}$ respectively. As leading coefficients of $\lambda(x_i)$ and $\lambda(x_j)$ are invertible, we know that the leading coefficient $\lambda(x_ix_j)$ is the product of leading coefficients  of $\lambda(x_i)$ and $\lambda(x_j)$. Moreover, we know that $x_ix_j=c_{ij}x_jx_i$, so, $\lambda(x_ix_j)=q_{ij}\lambda(x_jx_i)$, i.e. $\lambda(x_i)\lambda(x_j)=q_{ij}\lambda(x_j)\lambda(x_i)$. Then, the leading monomial of the right  term is equal to the leading monomial of the left term, i.e

\begin{eqnarray}
(\lambda_ix_1^{l_1}\cdots x_n^{l_n})(\lambda_jx_1^{t_1}\cdots x_n^{t_n})=q_{ij}(\lambda_jx_1^{t_1}\cdots x_n^{t_n})(\lambda_ix_1^{l_1}\cdots x_n^{l_n}).
\end{eqnarray}
So,
\begin{eqnarray}
\lambda_i\sigma^{l}(\lambda_j)(x_1^{l_1}\cdots x_n^{l_n})(x_1^{t_1}\cdots x_n^{t_n})=q_{ij}\lambda_j \sigma^t(\lambda_i)(x_1^{t_1}\cdots x_n^{t_n})(x_1^{l_1}\cdots x_n^{l_n}),
\end{eqnarray}
where $l=(l_1,\ldots,l_n)$ and $t=(t_1,\ldots,t_n)$.
\\
Therefore,
\begin{eqnarray*}
\lambda_i \lambda_j \left( \prod_{s<r}q_{rs}^{l_rt_s}\right) \equiv q_{ij}\lambda_i \lambda_j\left(\prod_{s<r} q_{rs}^{t_rl_s} \right) \quad m{o}d(N).
\end{eqnarray*}
Suppose that $i>j$. Since the images of $q_{rs}$, $n\geq r > s \geq 1$, in  $D^*/N$ are independent, we have for $r \neq s$
\begin{eqnarray}
l_rt_s=\delta_{ri}\delta_{sj}+t_rl_s.\label{3.2.3trabajo2}
\end{eqnarray}
If $l_p\neq 0$ for some $p \neq i,j$, then for every index $q \neq p$
\begin{eqnarray*}
l_qt_p=t_ql_p,
\end{eqnarray*}
and therefore, $t_q=t_pl_ql_p^{-1}$. In particular, when $q=i,j$
\begin{eqnarray*}
t_i &=& t_pl_il_p^{-1}, \\
t_j &=& t_pl_jl_p^{-1},
\end{eqnarray*}
so
\begin{eqnarray*}
l_it_j-l_jt_i&=& l_it_pl_jl_p^{-1}-l_jt_pl_il_p^{-1} \\
&=& 0,
\end{eqnarray*}
which contradicts \eqref{3.2.3trabajo2}. So $l_p =0$ for $p \neq i,j$. Similarly we can prove that $t_p=0$ if $p \neq i,j$.
\\

Hence, the leading coefficients of $\lambda(x_i)$ and $\lambda(x_j)$ have the form $\lambda_ix_i^{l_i}x_j^{l_j}$ and $\lambda_jx_i^{t_i}x_j^{t_j}$, where $l_it_j-l_jt_i=1$.
\\

By hypothesis there is a third variable $x_u$ such that $\lambda(x_u)\neq 0$. Applying the same argument to $(i,u)$ and $(j,u)$ must be $l_j, t_i$  are zero and that the leading coefficient of $\lambda(x_u)$ have the form 
\begin{eqnarray}
\lambda_ux_i^{d_i}x_u^{d_u}=\lambda_ux_j^{d_j}x_u^{d_u},
\end{eqnarray}
then, $d_i=d_j=0$. 
\\

So $l_it_j=1$ and $l_id_u=1$, therefore $l_i=t_j=d_u=\pm 1$.
\\

If we apply the corresponding argument for the smaller term of $\lambda(x_i)$, we obtain a similar result with some $\epsilon'=\pm 1$ for the smallest terms. Thus if $\epsilon=\epsilon'$ the theorem is proved.
\\

If $r=0$, then $\epsilon=\epsilon'=1$, since all exponents are positive.
\\

If $r>0$, consider $1\leq i \leq r$, so $\lambda(x_i)\in \Lambda^*$. Si $\epsilon'=-1$ and $\epsilon=1$, then  
\begin{eqnarray}
\lambda(x_i)=\lambda_i'x_i^{-1}+\sum \lambda_i''(s)x_i^{m_{i1}(s)}\cdots x_n^{m_{in}(s)}+\lambda_i''x_i,
\end{eqnarray}
where $\lambda_i''(s) \in D$, $\lambda_i'', \lambda_i' \in D^*$ and the sum is take over multi-indices 
\begin{eqnarray}
(m_{i1}(s),\ldots, m_{in}(s))\in \mathbb{Z}^n,
\end{eqnarray}
such that
\begin{eqnarray}
(0,\ldots,0,\overbrace{-1}^{i},0,\ldots,0)<(m_{i1}(s), \ldots,m_{in}(s)) < (0,\ldots, 0,\overbrace{1}^{i},0,\ldots,0).
\end{eqnarray}
As $\lambda_i'$ and $\lambda_i''$ are invertible, then $\lambda(x_i)$ can't be invertible, which is a contradiction. So, $\epsilon=\epsilon'$ and the theorem is proved.
\end{proof}

\begin{rem}
If $r<n$ and $\lambda(x_i)\neq 0$ for some $i>r$, then in theorem \ref{theorem2.1-2001}, $\epsilon=1$.
\end{rem}

Note that the first part of the above proof is taken from the original version of the theorem in \cite{va2001-2}.

\begin{cor}[\cite{va2001-2}]
If $\lambda(x_{r+1}), \ldots, \lambda(x_n)\neq 0$, in the situation of theorem \ref{theorem2.1-2001}, then $\lambda$ is an automorphism of $\Lambda$. In particular any injective endomorphism of $\Lambda$ is an automorphism.
\end{cor}

{\color{red}
\section{Automorphisms of skew $PBW$ extensions.}
}
In this section we use the ideas of Artamonov in \cite{va2001-2} to deduce some results for automorphisms of skew $PBW$ extensions. Therefore, we must first deduce some multiplication rules between monomials of $\Lambda$.
\\

\begin{defn}\label{definition1.1.1lezama}
Let $R$ and $A$ be rings, we say that $A$ is a skew $PBW$ extension of $R$ (also called $\sigma-PBW$ extension), if following conditions hold:
\begin{enumerate}
\item $R \subseteq A$.\label{condition1}
\item There exist finite elements $x_1,\ldots, x_n \in A$ such $A$ is a left $R-$free module with basis 
\[ Mon(A):=\{x^t=x_1^{t_1}\cdots x_n^{t_n}: t=(t_1,\ldots, t_n) \in \mathbb{N}^n\}. \]\label{condition2}
\item For every $1\leq i \leq n$ and $r \in R \setminus \{ 0\}$ there exists  $c_{i,r} \in R\setminus \{ 0\}$ such that
\begin{eqnarray}
x_ir-c_{i,r}x_i \in R \label{equation1.1.2lezama}.
\end{eqnarray}\label{condition3}
\item For every $1\leq i,j \leq n$ there exists $c_{j,i}\in R \setminus \{ 0\}$ such that
\begin{eqnarray}
x_jx_i-c_{j,i}x_ix_j \in R+Rx_1+\cdots Rx_n. \label{equation1.1.3lezama}
\end{eqnarray}\label{condition4}
\end{enumerate}
Under these conditions we will write $A=\sigma(R)\langle x_1,\ldots, x_n \rangle$.
\end{defn}

\begin{prop}[\cite{lezamalibro}]\label{proposition1.1.3lezama}
Let $A$ be a skew $PBW$ extension of $R$. Then, for every $1\leq i \leq n$, there exist an injective ring endomorphism $\sigma_i: R \longrightarrow R$ and a $\sigma_i-$derivation $\delta_i: R \longrightarrow R$ such that
\begin{eqnarray*}
x_ir=\sigma_i(r)x_i+\delta_i(r),
\end{eqnarray*}
for all $r \in R$.
\end{prop}

Let $N$ be the subgroup of the multiplicative group $R^*$, the invertible elements of $R$, generated by the derived subgroup $[R^*, R^*]$ and the set of all elements of the form $z^{-1}\sigma_i(z)$, where $z \in R^*$ and $\sigma_i$ is given by the proposition \ref{proposition1.1.3lezama} for $i=1,\ldots, n$.

\begin{defn}
Let $A$ be a skew $PBW$ extension of $R$.
\begin{enumerate}[(a)]
\item $A$ is quasi-conmutative if the condition (\ref{condition3}) and (\ref{condition4}) in definition \ref{definition1.1.1lezama} are replaced by
\begin{enumerate}
\item[(3')] For every $1\leq i \leq n$ and $r \in R \setminus \{ 0\}$ there exists $c_{i,r} \in R \setminus \{ 0\}$ such that
\begin{eqnarray}
x_ir=c_{i,r}x_i.
\end{eqnarray}
\item[(4')] For every $1\leq i,j \leq n$ there exists $c_{j,i} \in R \setminus \{ 0\}$ such that
\begin{eqnarray}
x_jx_i=c_{j,i}x_ix_j.
\end{eqnarray}
\end{enumerate}
\item $A$ es bijective if $\sigma_i$ is bijective for every $1\leq i \leq n$ and $c_{j,i}$ is invertible for any $1\leq i < j \leq n$.

\item $A$ es general if $c_{ij}$ is invertible for every $1\leq i,j \leq n$ and the images of all $c_{ij}$, $1 \leq  i < j \leq n$, are independent in the group $R^* /N$. 
\end{enumerate}
\end{defn}
{\color{red}
\subsection{Multiplication rules for quasi-commutative skew $PBW$ extensions. }
}

\begin{prop}\label{conmutacion}
Let $A$ be a quasi-conmutative skew $PBW$ extension where $c_{ij}$ is invertible for $1\leq i,j \leq n$. Then 
\begin{enumerate}
\item $x_ix_j^t=c_{ij}\sigma_j(c_{ij})\sigma_j^2(c_{ij})\cdots \sigma_j^{t-1}(c_{ij})x_j^tx_i$ for every $t \in \mathbb{N}$. \label{item1}

\item $\sigma_i^n(z)=zz^{-1}\sigma_i(z)\sigma_i(z)^{-1}\sigma_i^2(z)\sigma_i^{2}(z)^{-1}\cdots \sigma_i^{n-1}(z)^{-1}\sigma_i^{n}(z)$ for every $n \in \mathbb{N}$ and $z \in R^*$, i.e., $\sigma_i^n(z)=z\cdot d$, where $d \in N$. \label{item2}
\item \begin{eqnarray}
x_ix_j^t=(c_{ij}^t \cdot d)x_j^tx_i, 
\end{eqnarray}
for every $t \in \mathbb{N}$ and some $d \in N$.\label{item3}
\item \begin{eqnarray}
x_i^tx_j^s=(c_{ij}^{ts}\cdot d)x_j^sx_i^t,
\end{eqnarray}
for some $d \in N$ and every $t,s \in \mathbb{N}$. \label{item4}
\item \label{item5} \begin{eqnarray}
(x_1^{t_1}x_2^{t_2}\cdots x_n^{t_n})(x_1^{l_1}x_2^{l_2}\cdots x_n^{l_n})=\left( \prod_{i<j} c_{ji}^{t_js_i} \cdot d \right) x_1^{t_1+l_1}x_2^{t_2+l_2}\cdots x_n^{t_n+l_n},
\end{eqnarray}
for some $d \in N$.
\end{enumerate} 
\end{prop}

\begin{proof}
\begin{enumerate}
\item Let $t \in \mathbb{N}$ be, then
\begin{eqnarray*}
x_ix_j^t&=& c_{ij}x_jx_ix_j^{t-1}\\
&=& c_{ij}x_jc_{ij}x_jx_ix_j^{t-2}\\
&=& c_{ij}\sigma_j(c_{ij})x_j^2x_ix_j^{t-2}\\
&\vdots& \qquad \vdots \\
&=& c_{ij}\sigma_j(c_{ij})\cdots \sigma_j^{t-2}(c_{ij})x_j^{t-1}x_ix_j \\
&=&c_{ij}\sigma_j(c_{ij})\cdots \sigma_j^{t-2}(c_{ij})x_j^{t-1}c_{ij}x_jx_i \\
&=& c_{ij}\sigma_j(c_{ij})\cdots \sigma_j^{t-2}(c_{ij})\sigma_j^{t-1}(c_{ij})x_j^{t}x_i.
\end{eqnarray*}
\item This part is immediate since the $\sigma_i's$ are injective (proposition \ref{proposition1.1.3lezama}) and $z$ is invertible. Therefore, $\sigma_i^n(z)=z\cdot d$ for some $d \in N$.

\item By (\ref{item1})
\begin{eqnarray}
x_ix_j^t=c_{ij}\sigma_j(c_{ij})\cdots \sigma_j^{t-2}(c_{ij})\sigma_j^{t-1}(c_{ij})x_j^{t}x_i,
\end{eqnarray}
and by (\ref{item2})
\begin{eqnarray}
c_{ij}\sigma_j(c_{ij})\cdots \sigma_j^{t-2}(c_{ij})\sigma_j^{t-1}(c_{ij})=c_{ij}\cdot (c_{ij}\cdot d_1)\cdots (c_{ij}\cdot d_{t-1}),
\end{eqnarray}
for some $d_i \in N$, $1\leq i \leq t-1$. By the properties of N and the commutativity of $R^*/N$ we have
\begin{eqnarray}
c_{ij}\sigma_j(c_{ij})\cdots \sigma_j^{t-2}(c_{ij})\sigma_j^{t-1}(c_{ij}) \equiv c_{ij}^t \quad m{o}d(N),
\end{eqnarray}
therefore, $c_{ij}\sigma_j(c_{ij})\cdots \sigma_j^{t-2}(c_{ij})\sigma_j^{t-1}(c_{ij})=c_{ij}^t \cdot d$, for some $d \in N$. Hence, $x_ix_j^t=(c_{ij}^t \cdot d)x_j^tx_i$ with $d \in N$.

\item Suppose $s,t \in \mathbb{N}$, then
\begin{eqnarray*}
x_i^tx_j^s &=& x_i^{t-1}(c_{ij}^s\cdot d')x_j^sx_i\\
&=& \sigma_i^{t-1}(c_{ij}^s \cdot d')x_i^{t-1}x_j^sx_i \\
&=& \sigma_i^{t-1}(c_{ij}^s \cdot d')x_i^{t-2}(c_{ij}^s \cdot d')x_j^2x_i^2 \\
&=& \sigma_i^{t-1}(c_{ij}^s \cdot d')\sigma_i^{t-2}(c_{ij}^s\cdot d')x_i^{t-2}x_j^sx_i^2 \\
& \vdots & \qquad \vdots \\
&=& \sigma_i^{t-1}(c_{ij}^s \cdot d')\sigma_i^{t-2}(c_{ij}^s \cdot d') \cdot \sigma_i(c_{ij}^s \cdot d')x_ix_j^sx_i^{t-1} \\
&=& \sigma_i^{t-1}(c_{ij}^s \cdot d')\sigma_i^{t-2}(c_{ij}^s \cdot d') \cdot \sigma_i(c_{ij}^s \cdot d')(c_{ij}^s\cdot d')x_j^sx_i^t,
\end{eqnarray*}
with $d' \in N$.
\\

Using (\ref{item2}) and the definition of $N$ ,we have
\begin{eqnarray}
\sigma_i^{t-1}(c_{ij}^s \cdot d')\sigma_i^{t-2}(c_{ij}^s \cdot d') \cdot \sigma_i(c_{ij}^s \cdot d')(c_{ij}^s\cdot d')=(c_{ij}^{ts} \cdot d)x_j^sx_i^t,
\end{eqnarray}
for some $d \in N$.

\item The proof of this part is given by (\ref{item2}), (\ref{item3}), (\ref{item4}) and the properties of N.
\end{enumerate}
\end{proof}

{\color{red}
\subsection{Automorphisms of skew $PBW$ extensions.}
}
Let $A$ be a skew $PBW$ extension of $R$, we understand $End(A)$ as the ring endomorphisms of $A$ that fixed the elements of $R$ and $Aut(A)$ is the set of invertible elements of $End(A)$.\\

\begin{thm}\label{theorem3.2.1}
Suppose that $A=\sigma(R)\langle x_1, \ldots, x_n \rangle$ is a general quasi-commutative  extension , and $\lambda \in End(A)$. If $x_i$ is such that $\lambda(x_i) \neq 0$, the leading coefficient  (using the lexicographical order) of $\lambda(x_i)$ is invertible, and there exist almost three indices $i,j,u$ for which this happens, then there exist $\lambda_1,\ldots, \lambda_n \in R$ such that $\lambda(x_w)=\lambda_wx_w$, for $w=1,\ldots, n$.
\end{thm}

\begin{proof}
We consider the lexicographical order for the exponents of the monomials in A. We suppose that $\lambda(x_i)$ and $\lambda(x_j)$ are non-zero, and with leading monomial $\lambda_ix_1^{l_1}\cdots x_n^{l_n}$ and $\lambda_jx_1^{t_1}\cdots x_n^{t_n}$ respectively. As leading coefficients of $\lambda(x_i)$ and $\lambda(x_j)$ are invertible, we know that the leading coefficient $\lambda(x_ix_j)$ is the product of leading coefficients  of $\lambda(x_i)$ and $\lambda(x_j)$. Moreover, we know that $x_ix_j=c_{ij}x_jx_i$, so, $\lambda(x_ix_j)=c_{ij}\lambda(x_jx_i)$, i.e. $\lambda(x_i)\lambda(x_j)=c_{ij}\lambda(x_j)\lambda(x_i)$. Then, the leading monomial of the right  term is equal to the leading monomial of the left term, i.e

\begin{eqnarray}
(\lambda_ix_1^{l_1}\cdots x_n^{l_n})(\lambda_jx_1^{t_1}\cdots x_n^{t_n})=c_{ij}(\lambda_jx_1^{t_1}\cdots x_n^{t_n})(\lambda_ix_1^{l_1}\cdots x_n^{l_n}).
\end{eqnarray}
So,
\begin{eqnarray}
\lambda_i\sigma^{l}(\lambda_j)(x_1^{l_1}\cdots x_n^{l_n})(x_1^{t_1}\cdots x_n^{t_n})=c_{ij}\lambda_j \sigma^t(\lambda_i)(x_1^{t_1}\cdots x_n^{t_n})(x_1^{l_1}\cdots x_n^{l_n}),
\end{eqnarray}
where $l=(l_1,\ldots,l_n)$ and $t=(t_1,\ldots,t_n)$.
\\

By item (\ref{item5}) in the proposition  \ref{conmutacion} we obtain 
{\small
\begin{eqnarray*}
\lambda_i \sigma^l(\lambda_j)\left(\prod_{s<r} c_{rs}^{l_rt_s} \cdot d \right) x_1^{l_1+t_1} \cdots x_n^{l_n+t_n}=c_{ij}\lambda_j \sigma^t(\lambda_i) \left( \prod_{s<r} c_{rs}^{t_rl_s} \cdot d \right)  x_1^{l_1+t_1} \cdots x_n^{l_n+t_n}.
\end{eqnarray*}
}
Using the unique representation of the elements of $A$, and the property (\ref{item2}) of the proposition \ref{conmutacion} we deduce
\begin{eqnarray*}
\lambda_i \lambda_j \left( \prod_{s<r}c_{rs}^{l_rt_s}\right) \equiv c_{ij}\lambda_i \lambda_j\left(\prod_{s<r} c_{rs}^{t_rl_s} \right) \quad m{o}d(N).
\end{eqnarray*}
Given that $A$ is a general extension and suppose that $i>j$, we have
\begin{eqnarray}
l_rt_s=\delta_{ri}\delta_{sj}+t_rl_s.\label{3.2.3trabajo}
\end{eqnarray}
If $l_p\neq 0$ for some $p \neq i,j$, then for every index $q \neq p$
\begin{eqnarray*}
l_qt_p=t_ql_p,
\end{eqnarray*}
and therefore, $t_q=t_pl_ql_p^{-1}$. In particular, when $q=i,j$
\begin{eqnarray*}
t_i &=& t_pl_il_p^{-1}, \\
t_j &=& t_pl_jl_p^{-1},
\end{eqnarray*}
so
\begin{eqnarray*}
l_it_j-l_jt_i&=& l_it_pl_jl_p^{-1}-l_jt_pl_il_p^{-1} \\
&=& 0,
\end{eqnarray*}
which contradicts \eqref{3.2.3trabajo}. So $l_p =0$ for $p \neq i,j$. Similarly we can prove that $t_p=0$ if $p \neq i,j$.
\\

Hence, the leading coefficients of $\lambda(x_i)$ and $\lambda(x_j)$ have the form $\lambda_ix_i^{l_i}x_j^{l_j}$ and $\lambda_jx_i^{t_i}x_j^{t_j}$, where $l_it_j-l_jt_i=1$.
\\

By hypothesis, there is a third variable $x_u$ such that $\lambda(x_u)\neq 0$. Applying the same argument to $(i,u)$ and $(j,u)$ must be $l_j, t_i$  are zero and that the leading coefficient of $\lambda(x_u)$ have the form 
\begin{eqnarray}
\lambda_ux_i^{d_i}x_u^{d_u}=\lambda_ux_j^{d_j}x_u^{d_u},
\end{eqnarray}
then, $d_i=d_j=0$. 
\\

So $l_it_j=1$ and $l_id_u=1$, i.e, $l_i=t_j=d_u=1$.
\\

Therefore, for every $1\leq i \leq n$, $\lambda(x_i)=\lambda_ix_i$, with $\lambda_i \in R^*$ if $\lambda(x_i) \neq 0$.
\end{proof}

\begin{cor}
If $A$ is a skew $PBW$ extension of $R$, $\lambda \in End(A)$, in the situation of Theorem 2.1, and every $\lambda(x_i)\neq 0$, $1\leq i \leq n$, then $\lambda$ is an automorphism of $A$. In particular, any injective endomorphism of $A$ is an automorphism.
\end{cor}

\begin{proof}
We observe that
\begin{eqnarray}
\lambda(x_1^{l_1}x_2^{l_2}\cdots x_n^{l_n})=\left(\prod_k \lambda_k^{l_k}\cdot u \right) x_1^{l_1}\cdots x_n^{l_n},
\end{eqnarray}
where $\lambda_k$ is given by the theorem \ref{theorem3.2.1} and for some $u \in N$.
\\

By hypothesis, any $\lambda_i \in R^*$ and $N \subseteq R^*$, then $\left(\prod\limits_k \lambda_k^{l_k}\cdot u \right) \in R^*$, and hence we can consider $\lambda(x_1^{l_1}\cdot x_n^{l_n})=d \cdot x_1^{l_1}\cdots x_n^{l_n}$, for some $d \in R^*$.
\\

If $\lambda(\sum\limits_{l\in \mathbb{N}^n}a_lx^l)=0$, then $\sum\limits_{l\in \mathbb{N}^n}a_ld_lx^l=0$, where $d_l$ is invertible. By the unique representation of each element of $A$ we conclude that each $a_ld_l=0$ for $l \in \mathbb{N}^n$. As each $d_l$ is invertible, then $a_l=0$ for any $l$. i.e., $\sum\limits_{l\in \mathbb{N}^n}a_lx^l=0$, proving that $\lambda$ is an injective endomorphism.
\\

Similarly, if $p=\sum\limits_{l\in \mathbb{N}^n}a_lx^l$ then $p'=\sum\limits_{l\in \mathbb{N}^n}a_ld_l^{-1}x^l$ satisfies $\lambda(p')=p$, proving that $\lambda$ is a surjective endomorphism.  
\end{proof}

\begin{defn}[\cite{lezamalibro}]
Let $A$ be a skew $PBW$ extension of $R$ with endomorphisms $\sigma_i$, $1\leq i \leq n$, as in Proposition \ref{proposition1.1.3lezama}.
\begin{enumerate}
\item For $\alpha=(\alpha_1, \ldots, \alpha_n) \in \mathbb{N}^n$, $\sigma^\alpha:=\sigma_1^{\alpha_1} \cdots \sigma_n^{\alpha_n}$, $|\alpha|=\alpha_1+\cdots \alpha_n$. If $\beta=(\beta_1, \ldots, \beta_n)\in \mathbb{N}^n$, then $\alpha+\beta=(\alpha_1+\beta_1,\ldots, \alpha_n+\beta_n)$.
\item For $X=x^{\alpha} \in Mon(A)$, $exp(X):=\alpha$ and $deg(X):=|\alpha|$.
\item Let $0\neq f \in A$, $t(f)$ is the finite set of terms that conform $f$, i.e., if $f=c_1X_1+\cdots +c_tX_t$, with $X_i \in Mon(A)$ and $c_i \in R \setminus \{0\}$, then $t(f):=\{c_1X_1,\ldots, c_tX_t\}$. \label{3definicion1.1.7}
\item Let $f$ be as in (\ref{3definicion1.1.7}), then $deg(f):=\max\{ deg(X_i)\}_{i=1}^t$.
\end{enumerate}
\end{defn}

\begin{thm}[\cite{lezama},\cite{lezamalibro}]\label{proposition1.2.1lezama}
Let $A$ be a skew $PBW$ extension of $R$. Then, $A$ is a filtered ring with filtration given by
\begin{eqnarray}
F_m:=  \begin{cases}
R, \qquad \qquad \qquad \qquad \qquad if \quad m=0,\\
\{ f \in A: deg(f)\leq m\}, \,\,\, \quad if \quad m\geq 1
\end{cases}
\end{eqnarray}
and the corresponding graded ring $Gr(A)$ is a quasi-commutative skew $PBW$ extension of $R$. Moreover, if $A$ is bijective, then $Gr(A)$ is a quasi-commutative bijective skew $PBW$ extension of $R$.
\end{thm}

\begin{thm}\label{theorem3.2.3}
Let $A$ be a skew $PBW$ extension of $R$ and $\lambda \in End(A)$ a filtered endomorphism, such that
\begin{enumerate}
\item $A$ is general,
\item for each $\lambda(x_i)\neq 0$, $\lambda(x_i)\notin F_0(A)$ and the leading coefficient (according to the lexicographical order) is invertible, and there exists at least three indices $i,j,u$ with this property.
\end{enumerate}
Then, there exist $\lambda_1,\ldots,\lambda_n \in R$ and $a_{10},a_{20}, \ldots,a_{n0}\in R$ such that $\lambda(x_w)=a_{w0}+\lambda_wx_w$, for $w=1,\ldots,n$.
\end{thm}

\begin{proof}
As $\lambda$ is a filtered endomorphism, we induce a graded endomorphism $Gr(\lambda):Gr(A) \longrightarrow Gr(A)$. We observe that $Gr(\lambda)(x_i+F_0)\neq 0$, for those $x_i$ with $\lambda(x_i)\neq 0$. So, all the hypothesis of the Theorem \ref{theorem3.2.1} are satisfy  for $Gr(A)$ and $Gr(\lambda)$, therefore $Gr(\lambda)(x_i)=\lambda_ix_i+F_0$ for some $\lambda_i \in R$.
\\

Finally, as $\lambda(x_i)+F_0=\lambda_ix_i+F_0$, then $\lambda(x_i)=\lambda_ix_i+a_{i0}$ for some $a_{i0}\in R$ and for any $i=1,\ldots, n$. 
\end{proof}

\begin{cor}\label{corolario3.2.2}
If $A$ is skew $PBW$ extension of $R$, $\lambda \in End(A)$, in the situation of the theorem \ref{theorem3.2.3}, and every $\lambda(x_i)\neq 0$, then  $\lambda$ is an automorphism of $A$. In particular, if $\lambda$ is an injective endomorphism, $\lambda$ is an automorphism. 
\end{cor}

\begin{proof}
$Gr(\lambda)(x_i+F_0)\neq 0$ for every $i$, by the corollary \ref{corolario3.2.2}, $Gr(\lambda)$ is an automorphism of $Gr(A)$, since the filtration is positive, then $\lambda$ is a strict automorphism of $A$ (see \cite{lezamahomologica}).
\end{proof}

{\color{red}
\subsection{Automorphisms of skew $PBW$ extension on  left Ore domain }
}

In this subsection we give a description of the automorphism group of a skew $PBW$ extension of a left Ore domain $R$. \\

\begin{thm}[\cite{lezamalibro}]\label{theoremquasipolynomial}
Let $A$ be a quasi-commutative skew $PBW$ extension of a ring $R$. Then,
\begin{enumerate}
\item $A$ is isomorphic to an iterated skew polynomial ring of endomorphism type.
\item If $A$ is bijective, then each endomorphism is bijective.
\end{enumerate}
\end{thm}

For the proof, we observe that we give a ring isomorphism with the iterated skew polynomial ring $B:=R[z_1; \theta_1]\cdots [z_n; \theta_n]$, where
\begin{eqnarray}
\begin{cases}
\theta_1=\sigma_1, \\
\theta_j: R[z_1; \theta_1]\cdots [z_{j-1},\theta_{j-1}] \longrightarrow R[z_1; \theta_1]\cdots [z_{j-1},\theta_{j-1}], \\
\theta_j(z_i)=c_{ji}z_i \text{ para } 1\leq i < j \leq n, \quad \theta_j(r)=\sigma_j(r) \text{ para }r \in R,
\end{cases}
\end{eqnarray}
where $\sigma_i$ is given by the Proposition \ref{proposition1.1.3lezama}.

\begin{thm}\label{theorem1.5.6lezama}
Let $R$ be a ring and $S \subset R$ a multiplicative subset such that 
\begin{enumerate}
\item $S^{-1}R$ exists.
\item $\sigma(S)\subseteq S$.
\end{enumerate}
then 
\begin{eqnarray}
S^{-1}(R[x;\sigma,\delta])\cong (S^{-1}R)[x;\overline{\sigma}, \overline{\delta}],
\end{eqnarray}
with
\begin{eqnarray}
& S^{-1}R & \xlongrightarrow{\overline{\sigma}} S^{-1}R\nonumber \\
& \frac{a}{s} & \longmapsto \frac{\sigma(a)}{\sigma(s)} \nonumber ,
\end{eqnarray}
\begin{eqnarray}
& S^{-1}R & \xlongrightarrow{\overline{\delta}} S^{-1}R\nonumber \\
& \frac{a}{s} & \longmapsto \frac{\delta(s)}{\sigma(s)}\frac{a}{s}+\frac{\delta(a)}{\sigma(s)}. \nonumber 
\end{eqnarray}
Moreover, if $RS^{-1}$ exists, then
\begin{eqnarray}
(R[x; \delta])S^{-1}\cong (RS^{-1})[x;\overline{\delta}].
\end{eqnarray}
\end{thm}

According the proof in \cite{lezamalibro}, the ring homomorphism $\psi:R[x;\sigma,\delta]\longrightarrow R[x;\sigma,\delta]$, in the university property for $S^{-1}R$, is give by $\psi(\sum\limits_i a_ix^i)=\sum\limits_i\frac{a_i}{1}x_i$.

\begin{cor}\label{corollary1.5.8lezama}
Let $R$ be a ring, $S \subset R$ a multiplicative subset and
\[ A:=R[x_1;\sigma_1,\delta_1]\cdots [x_n;\sigma_n,\delta_n]\]
the iterated skew polynomial ring sucht that
\begin{enumerate}
\item $S^{-1}R$ exists,
\item $\sigma_i(S)\subseteq S$, for each $1\leq i \leq n$.
\end{enumerate}
Then
\begin{eqnarray}
S^{-1}(R[x_1;\sigma_1, \delta_1]\cdots [x_n;\sigma_n, \delta_n]) \cong (S^{-1}R)[x_1;\overline{\sigma_1}, \overline{\delta_1}]\cdots [x_n;\overline{\sigma_n}, \overline{\delta_n}],
\end{eqnarray}
with
\begin{eqnarray*}
& S^{-1}R  & \xlongrightarrow{\overline{\sigma_i}} S^{-1}R \\
& \frac{a}{s} & \longmapsto \frac{\sigma_i(a)}{\sigma_i(s)}, 
\end{eqnarray*}

\begin{eqnarray*}
& S^{-1}R  & \xlongrightarrow{\overline{\delta_i}} S^{-1}R \\
& \frac{a}{s} & \longmapsto \frac{\delta_i(s)}{\sigma_i(s)}\frac{a}{s}+\frac{\delta_i(a)}{\sigma_i(s)}.
\end{eqnarray*}
\end{cor} 
Applying of iterated form the construction in the Theorem \ref{theorem1.5.6lezama}, we obtain that the function $\psi: A \longrightarrow(S^{-1}R)[x_1; \overline{\sigma_1}, \overline{\delta_1}]\cdots [x_n; \overline{\sigma_n}, \overline{\delta_n}]$, in the universal property of $S^{-1}A$ is give by $\psi(\sum\limits_t a_tx^t)=\sum\limits_t\frac{a_t}{1}x^t$.

\begin{prop}\label{proposition3.3.1}
Let $A=R[x_1,\theta_1]\cdots [x_n,\theta_n]$ be a iterated skew polynomial ring of endomorphism type, where $\theta_i|_R$ is a injective endomorphism of $R$ for $1\leq i \leq n$, and $\theta_i(x_j)=c_{ij}x_j$, for every $1\leq j < i \leq n$ with $c_{ij}$ left invertible. Then $A$ is a quasi-commutative skew $PBW$ extension of $R$ with endomorphisms $\sigma_i=\theta_i|_R$.
\end{prop}

\begin{proof}
If $n=1$, we have that $A=R[x_1,\theta_1]$ is a quasi-commutative skew $PBW$ extension of $R$ (see \cite{lezamalibro}).
\\

Suppose that $A'=R[x_1,\theta_1]\cdots [x_{n-1},\theta_{n-1}]$ is a quasi-commutative skew $PBW$ extension, and consider $A=R[x_1,\theta_1]\cdots [x_n,\theta_n]=A'[x_n,\theta_n]$ a quasi-commutative skew $PBW$ extension of $A'$. So

\begin{enumerate}
\item $R \subseteq A' \subseteq A$.
\item All element $p \in A$ can write like
\begin{eqnarray}
p=\sum_j^{n}a_jx_n^j,
\end{eqnarray}
with $a_j \in A'$.
\\

By hypothesis $a_j=\sum\limits_l a_{jl}x^l$, with $a_{jl}\in R\setminus \{0\}$ and $l=(l_1,\ldots, l_{n-1})$. So
\begin{eqnarray}
p=\sum_{j=0}^n \sum_l a_{jl}x^lx_n^j,
\end{eqnarray} 
whence, any element $p \in A$ can write with linear combination of $Mon(A):=Mon\{ x_1,\ldots,x_n\}$.
\\

We see that $Mon(A)$ is a $R-$basis for $A$. We suppose that $0=\sum\limits_t a_tx^t$ with $t=(t_1,\ldots,t_n) \in \mathbb{N}^n$ and $a_t \in R$, then we have
\begin{eqnarray}
0 &= & \sum_{i=0}^s \left( \sum_{\{t: t_n=i\}}a_tx_1^{t_1}\cdots x_{n-1}^{t_{n-1}}\right)x_n^i, 
\end{eqnarray}
where $s$ is the largest exponent appearing for $x_n$.
\\

Since $A=\sigma(A')\langle x_n \rangle$, $\sum\limits_{\{ t:t_n=i\}}a_tx_1^{t_1}\cdots x_{n-1}^{t_{n-1}}\in A'$, and the representation of zero is unique, and consider $A$ with $A'-$module, we have that $\sum\limits_{\{ t:t_n=i\}}a_tx_1^{t_1}\cdots x_{n-1}^{t_{n-1}}=0$ for every $i$.
\\

Finally, as $A'=\sigma(R)\langle x_1,\ldots,x_{n-1}\rangle$ then $\sum_{\{ t:t_n=i\}}a_tx_1^{t_1}\cdots x_{n-1}^{t_{n-1}}=0$ implies that each $a_t=0$, therefore the element zero has a unique representation in the elements $Mon(A)$. So, $Mon(A)$ is a $R-$basis for $A$.

\item If $r \in R \setminus \{ 0 \}$ then $x_ir=\theta_i(r)x_i$, but by hypothesis we know that $\theta_i(r) \in R \setminus \{ 0\}$. So there exists $c_{i,r}:=\theta_i(r)$ such that $x_ir-c_{ir}x_i=0 \in R$.

\item If $j<i$, $x_ix_j=\theta_i(x_j)x_i=c_{ij}x_jx_i$, as $c_{ij}$ is left invertible, then $c_{ji}x_ix_j=x_jx_i$, where $c_{ji}$ is the inverse of $c_{ij}$. Next, we have found $c_{ij} \in R\setminus \{ 0\}$ such that $x_ix_j-c_{ij}x_jx_i=0 \in R+Rx_1+\cdots Rx_n$.
\end{enumerate}
So we have that $A$ is a quasi-commutative skew $PBW$ extension of $R$.
\\

ultimately, we see that $x_ir=\theta_i(r)x_i$ for each $i$, which tells us, by the unique representation of elements in $A$, that $\sigma_i(r)=\theta_i(r)$ for every $r \in R$.

\end{proof}
From now on assume that $S:=R \setminus \{ 0\}$.
\begin{thm}\label{theorem3.3.3}
Let $A$ be a quasi-commutative skew $PBW$ extension over a left Ore Domain $R$ with endomorphisms $\sigma_i$. Then $S^{-1}A \cong \sigma(S^{-1}R)\langle x_1,\ldots, x_n \rangle$ and is a quasi-commutative skew $PBW$ extension of $S^{-1}R$, where the endomorphisms $\overline{\sigma_i}$ for $\sigma(S^{-1}R)\langle x_1,\ldots, x_n \rangle$ are
\begin{eqnarray}
\overline{\sigma_i}: & S^{-1}R& \longrightarrow S^{-1}R \nonumber \\
&\frac{a}{s}& \longmapsto  \frac{\sigma_i(a)}{\sigma_i(s)},
\end{eqnarray}
and
\begin{eqnarray}
\psi: & A & \longrightarrow S^{-1}A \nonumber \\
& \sum\limits _t a_tx^t & \longmapsto \sum_t \frac{a_t}{1}x^t,
\end{eqnarray}
is the homomorphism in the universal property of $S^{-1}A$.
\end{thm}

\begin{proof}
Since $R$ is a left Ore domain, then $S=R \setminus \{ 0\}$ is a Ore set, and therefore of $A$.
\\

$A$ is a quasi-commutative then by the Theorem \ref{theoremquasipolynomial}, $A \cong R[x_1; \theta_1]\cdots [x_n; \theta_n]$ such that
\begin{eqnarray}
\begin{cases}
\theta_1=\sigma_1, \\
\theta_j: R[x_1, \theta_1]\cdots [x_{j-1}; \theta_{j-1}] \longrightarrow R[x_1, \theta_1]\cdots [x_{j-1}; \theta_{j-1}] \\
\theta_j(x_i)=c_{ji}x_i \text{ for } 1 \leq i <j \leq n, \quad \theta_j(r)=\sigma_j(r)\text{ for }r \in R,
\end{cases}
\end{eqnarray}
where the $\sigma_i$ are given for the proposition \ref{proposition1.1.3lezama}.
\\

Observe that for all $r \in S$, $\theta_i(r)=\sigma_i(r)$. So, by the Theorem \ref{corollary1.5.8lezama}, we can consider 
\begin{eqnarray}
S^{-1}(R[x_1;\theta_1]\cdots [x_n; \theta_n]) \cong (S^{-1}R)[x_1; \overline{\theta_1}]\cdots [x_n; \overline{\theta_n}],
\end{eqnarray}
with
\begin{eqnarray}
& S^{-1}R & \xlongrightarrow{\overline{\theta_i}}S^{-1}R \nonumber \\
& \frac{a}{s} & \longmapsto \frac{\theta_i(a)}{\theta_i(s)}=\frac{\sigma_i(a)}{\sigma_i(s)}.
\end{eqnarray}

Moreover, $\overline{\theta_i}\left(\frac{a}{s}\right)=\overline{\sigma_i} \left( \frac{a}{s}\right)$ for each $\frac{a}{s}\in S^{-1}R$, where $\overline{\sigma_i}$ is given by the universal property of $S^{-1}R$ with homomorphism $\psi '$.

\begin{eqnarray}
\begin{diagram}
\node{R} \arrow{s,l}{\sigma_i} \arrow[2]{e,t}{\psi'} \node{} \node{S^{-1}R} \arrow[2]{sw,r,..}{\overline{\sigma_i}}\\
\node{R}\arrow{s,l}{\psi'} \\
\node{S^{-1}R}
\end{diagram}
\end{eqnarray}
We have that $\overline{\sigma_i}$ is a injective homomorphism, since $R$ has no zero divisors, and $\sigma_i$ is injective. So $\overline{\theta_i}$ is a injective homomorphisms of $S^{-1}R$.
\\

Therefore, we have an  iterated skew polynomial ring of endomorphism type
\[ (S^{-1}R)[x_1; \overline{\theta_1}]\cdots [x_n; \theta_n],\]

such that $\overline{\theta_i}$ is a ring injective endomorphism of $S^{-1}R$ for $1\leq i \leq n$. Hence, using the Proposition \ref{proposition3.3.1}, $(S^{-1}R)[x_1; \overline{\theta_1}]\cdots [x_n; \theta_n]$ is a quasi-commutative skew $PBW$ extension form $\sigma(S^{-1}R)\langle x_1,\ldots, x_n \rangle$, with endomorphisms $\overline{\sigma_i}$.
\\

Finally, $\sigma(S^{-1}R)\langle x_1, \ldots, x_n \rangle \cong S^{-1}A$, where $\psi$ is given in the proof  Corollary \ref{corollary1.5.8lezama} and coincides with that given in this theorem.
\end{proof}

For the next results is necessary to assume that if $A$ is a quasi-commutative extension, then $S^{-1}A$ is a general quasi-commutative extension of $S^{-1}R$, since, I could not test that, if they are independent in $R$ then  are independent in $(S^{-1}R)^*/N$.

\begin{thm}\label{theorem3.3.4}
Suppose  that $A$ is a quasi-commutative skew $PBW$ extension over a left Ore domain $R$, such that $S^{-1}A$ is a general quasi-commutative extension, $\lambda \in End(A)$ , and there exists at least three indices $i,j,t$, with $\lambda(x_i), \lambda(x_j), \lambda(x_t) \neq 0$. Then there exist $\lambda_1,\ldots, \lambda_n \in R$ such that $\lambda(x_w)=\lambda_wx_w$, for every $1\leq w \leq n$.
\end{thm}

\begin{proof}
Consider the Ore set $S$, then $S^{-1}R$ exist, and by the theorem \ref{theorem3.3.3} we can consider $S^{-1}A$. Let $\lambda \in End(A)$, then $\lambda(r)=r$ for all $r \in R$, so we find $\overline{\lambda} \in End(S^{-1}A)$,

\begin{eqnarray}
\begin{diagram}
\node{A} \arrow{s,l}{\lambda} \arrow[2]{e,t}{\psi} \node{} \node{S^{-1}A} \arrow[2]{sw,r,..}{\overline{\lambda}}\\
\node{R}\arrow{s,l}{\psi} \\
\node{S^{-1}A}
\end{diagram}
\end{eqnarray}
Furthermore, $\psi$ is injective, because $S$ is not zero divisor in $A$, wherefore we consider $A \hookrightarrow S^{-1}A$. Then, suppose that $\lambda(x_i)=\sum{a_{it}x^t}$ for every $1\leq i \leq n$, $\overline{\lambda}(x_i)=\psi(\lambda(x_i))=\sum \frac{a_{it}}{1}x^t$. Consequently, if $\overline{\lambda}(x_i)=0$ then $\lambda(x_i)=0$.
\\

Next, There are at least three $i,j,t$ such that $\overline{\lambda}(x_i),\overline{\lambda}(x_j),\overline{\lambda}(x_t)\neq 0$, as $S^{-1}A$ is a quasi-commutative general skew $PBW$ extension of $S^{-1}R$, then by the Theorem \ref{theorem3.2.1}, $\overline{\lambda}(x_i)=\frac{a_i}{1}x_i$ for every $1\leq i \leq n$ ($a_i:=a_{ii}$). Since $\psi$ is injective, we conclude that $a_{it}=0$ for $t \neq i$,  and therefore $\lambda(x_i)=a_ix_i$ for every $i=1,\ldots, n$. 
\end{proof}

\begin{cor}\label{corollary3.3.2}
If $A$ and $\lambda \in End(A)$, in the situation of the Theorem \ref{theorem3.3.4}, and $\lambda$ is injective, then there are $\lambda_1,\ldots, \lambda_n \in R \setminus \{ 0\}$, such that
\begin{eqnarray}
\lambda(x_w)=\lambda_wx_w \text{ para cada }w=1,\ldots , n. \label{equation3.3.14}
\end{eqnarray}
In particular, if $\lambda \in Aut(A)$ then $\lambda$ has the form \eqref{equation3.3.14}.
\end{cor}

\begin{cor}
If $A$ is a skew $PBW$ extension over left Ore domain $R$ such that $c_{ij}$, $1\leq i < j \leq n$, are independent in $(S^{-1}R)^*/N$, and $\lambda$ is a filtered endomorphism, such that if $\lambda(x_i)\neq 0$, $\lambda(x_i) \notin F_0(A)$ then there exists $\lambda_{01},\ldots, \lambda_{0n},\lambda_1,\ldots, \lambda_n \in R$ such that
\begin{eqnarray}
\lambda(x_i)=\lambda_{0i}+\lambda_ix_i, \quad 1\leq i \leq n.
\end{eqnarray}
In particular, if $\lambda$ is injective, $\lambda_1,\ldots,\lambda_n \in R \setminus \{ 0\}$.
\end{cor}

{\color{red}
\section{Skew quantum polynomials}
}

\begin{defn}[\cite{lezamalibro}]\label{quantumskew}
Let $R$ be a ring with a fixed matrix of parameters $Q:=(q_{ij}) \in Mat_n(R)$, $n \geq 2$, such that $q_{ii}=1=q_{ij}q_{ji}=q_{ji}q_{ij}$ for every $1\leq i,j \leq n$, and suppose also that it is given a system $\sigma_1,\ldots, \sigma_n$ of automorphisms of $R$. The ring of skew quantum polynomials over $R$, denoted by $R_{Q, \sigma}[x_1^{\pm 1}, \ldots, x_r^{\pm 1},x_{r+1}, \ldots, x_n]$ or $\mathcal{O}_{Q,\sigma}$, is defined as following:
\begin{enumerate}
\item $R \subseteq R_{Q, \sigma}[x_1^{\pm 1}, \ldots, x_r^{\pm 1},x_{r+1}, \ldots, x_n] $.
\item $R_{Q, \sigma}[x_1^{\pm 1}, \ldots, x_r^{\pm 1},x_{r+1}, \ldots, x_n]$ is a free left $R-$module with basis
\begin{eqnarray}
\{ x_1^{t_1} \cdots x_n^{t_n}: t_i \in \mathbb{Z} \text{ for }1\leq i \leq r \text{ and }t_i \in \mathbb{N} \text{ for }r+1 \leq i \leq n\}.
\end{eqnarray}
\item The elements $x_1,\ldots, x_n$ satisfy the defining relations
\begin{eqnarray}
x_ix_i^{-1}=& 1 &=x_i^{-1}x_i, \quad 1\leq i \leq r, \nonumber \\
x_jx_i&=& q_{ji}x_ix_j, \quad 1\leq i,j ,\leq n, \nonumber \\
x_ir &=& \sigma_i(r)x_i,\quad r \in R, \quad 1\leq i,j \leq n. \nonumber
\end{eqnarray}
\end{enumerate}
\end{defn}
$R_{Q, \sigma}[x_1^{\pm 1}, \ldots, x_r^{\pm 1},x_{r+1}, \ldots, x_n]$ can be viewed as a localization of a skew $PBW$ extension. In fact, we have the quasi-commutative bijective skew $PBW$ extension
\begin{eqnarray*}
A:=\sigma(R)\langle x_1,\ldots, x_n \rangle,
\end{eqnarray*}
with $x_ir =\sigma_i(r)x_i$ and $x_jx_i=q_{ji}x_ix_j$, $1\leq i,j \leq n$; if we set
\begin{eqnarray*}
S:=\{ rx^t:r \in R^*,x^t \in Mon\{ x_1,\ldots, x_r\}\},
\end{eqnarray*}
then $S$ is a multiplicative subset of $A$ and
\begin{eqnarray*}
S^{-1}A \cong R_{Q,\sigma}[x_1^{\pm 1}, \ldots, x_{r}^{\pm 1},x_{r+1}, \ldots, x_n].
\end{eqnarray*}

We define $N$ as the subgroup of the multiplicative group $R^*$ generated by $[R^*, R^*]$ and the elements form $z^{-1}\sigma_i(z)$ for $z \in R^*$ and $i=1, \ldots, n$. This subgroup is normal in $R^*$ and $R^*/N$ is abelian group.

\begin{prop}\label{proposition5.2.2}
If $\mathcal{O}_{Q,\sigma}$ is a skew  quantum polynomials over $R$, then
\begin{enumerate}
\item $x_i^{-1}r=\sigma_i^{-1}(r)x_i^{-1}$ for every $r \in R$ and $1\leq i \leq r$.
\item $x_ix_j^{-1}=\sigma_j^{-1}(q_{ij}^{-1})x_j^{-1}x_i$ for $1\leq i \leq n$ and $1\leq j \leq r$.
\item $\sigma_i^n(z)=zz^{-1}\sigma_i(z)\sigma_i(z)^{-1}\cdots \sigma_i^{n-1}(z)^{-1}\sigma_i^n(z)$ for all $n\in \mathbb{N}$ and $z \in R^*$, moreover $\sigma_i(z)=z \cdot d$ with $d \in N$.
\item $\sigma_i^{-n}(z)=\sigma_i^{-n}(z)\sigma_i^{-n+1}(z)^{-1}\sigma_i^{-n+1}(z)\cdots \sigma_i^{-1}(z)^{-1}\sigma_i^{-1}(z)z^{-1}z$ for $n\in \mathbb{N}$ and $z\in R^*$. moreover $\sigma_i^{-n}(z)=z \cdot u$ with $u \in N$.
\item $x_ix_j^n=\left(q_{ij}^n \cdot d \right)x_j^nx_i$, with $d \in N$ and for every $n \in \mathbb{Z}$.
\item $x_i^tx_j^s=\left( q_{ij}^{ts} \cdot u \right)x_j^sx_i^t$ for some $u \in N$, and for every $t,s \in \mathbb{Z}$.
\item $(x_1^{t_1}\cdots x_n^{t_n})(x_1^{s_1}\cdots x_n^{s_n})=\left( \prod\limits_{i<j}q_{ji}^{t_js_i}\cdot u \right) x_1^{t_1+s_1}+\cdots x_n^{t_n+s_n}$, for some $u \in N$.
\end{enumerate}
\end{prop}

\begin{proof}
The proof is similar to the Proposition \ref{conmutacion}.
\end{proof}
\begin{defn}

A ring of skew quantum polynomials $\mathcal{O}_{Q,\sigma}$ of $R$, is general when the multiparameters $q_{ij}$, $1\leq i < j \leq n$, are independent in $R^*/N$.
\end{defn}

{
\color{red}
\subsection{Automorphisms for the  skew quantum polynomial rings.}
}
We understand $End(\mathcal{O}_{Q,\sigma})$ as the ring endomorphisms of $\mathcal{O}_{Q,\sigma}$ acting identically on $R$, and $Aut(\mathcal{O}_{Q,\sigma})$ is the set of the elements invertible in $End(\mathcal{O}_{Q,\sigma})$.\\

\begin{thm}
Let $\mathcal{O}:=\mathcal{O}_{Q,\sigma}$ be the general ring of skew quantum polynomials of $R$ with $r=0$, $n\geq 3$ and $\lambda \in End(\mathcal{O})$. If for every $1\leq i \leq n$, $\lambda(x_i)\neq 0$, and the leading coefficient is invertible, then $\lambda$ is an automorphism of $\mathcal{O}$.
\end{thm}

\begin{proof}
This is a particular case of the Theorem  \ref{theorem3.2.1}.
\end{proof}

\begin{thm}\label{theorem4.3.2}
Suppose that $\lambda \in End(\mathcal{O})$ and
\begin{enumerate}
\item for every $\lambda(x_i)\neq 0$, the leading and smallest terms are invertible,
\item there are at least three indices $1\leq i,j,t \leq n$ such that $\lambda(x_i),\lambda(x_j),\lambda(x_t)\neq 0$.
\end{enumerate}
Then there are elements $\lambda_1,\ldots, \lambda_n \in R$ and $\epsilon= \pm 1$ such that $\lambda_1,\ldots, \lambda_r \neq 0$ and
\begin{eqnarray}
\lambda(x_w)=\lambda_wx_w^\epsilon \quad w=1,\ldots, n.
\end{eqnarray}
\end{thm}

\begin{proof}
Similar to Theorem \ref{theorem2.1-2001}. 
\end{proof}

\begin{cor}
If $\lambda(x_{r+1}), \ldots ,\lambda(x_n) \neq 0$, in the situation of the theorem \ref{theorem4.3.2}, then $\lambda$ is a automorphism of $\mathcal{O}$. In particular every injective endomorphism of $\mathcal{O}$ is a automorphism.
\end{cor}

{\color{red}
\subsection{Automorphisms of quantum skew polynomial rings over Ore domains.}
}
We consider $\mathcal{O}_{Q,\sigma}$ y $S:=R \setminus \{ 0\}$, where $R$ is an Ore domain, we want to see that
\[ S^{-1}\mathcal{O}_{Q,\sigma}\cong (S^{-1}R)_{Q,\overline{\sigma}}[x_1^{\pm 1},\ldots, x_r^{\pm 1}, x_{r+1},\ldots, x_n] ,\]

where $\overline{\sigma_i}: S^{-1}R \longrightarrow S^{-1}R$ is given by universal property of $S^{-1}R$.
\begin{eqnarray}
\begin{diagram}
\node{R} \arrow{s,l}{\sigma_i} \arrow[2]{e,t}{\psi} \node{} \node{S^{-1}R} \arrow[2]{sw,r,..}{\overline{\sigma_i}}\\
\node{R}\arrow{s,l}{\psi} \\
\node{S^{-1}R}
\end{diagram},
\end{eqnarray}
Furthermore, as $\psi \sigma_i$ is injective, then $\overline{\sigma_i}$ is an injective endomorphism. Lets see that it is surjective too. Consider $\frac{a}{s} \in S^{-1}R$, there are $a' \in R$ and $s' \in S$, such that $\sigma_i(a')=a$ and $\sigma_i(s')=s$, so
\begin{eqnarray}
\overline{\sigma_i}(\frac{a'}{s'})=\frac{a}{s}.
\end{eqnarray}

Let us now that $S^{-1}\mathcal{O}_{Q,\sigma}\cong (S^{-1}R)_{Q,\overline{\sigma}}[x_1^{\pm 1},\ldots,x_r^{\pm 1},x_{r+1},\ldots,x_n]$. We consider the ring homomorphism

\begin{eqnarray}
\varphi: &\mathcal{O}_{Q,\sigma}& \longrightarrow (S^{-1}R)_{Q,\overline{\sigma}}[x_1^{\pm 1},\ldots,x_r^{\pm 1},x_{r+1},\ldots,x_n]\nonumber \\
&\sum a_tx^t & \longmapsto \sum \frac{a_t}{1}x^t,
\end{eqnarray}

\begin{enumerate}

\item If $r \in R \setminus \{ 0\}$, then  $\varphi(r)=\frac{r}{1} $ is invertible, therefore \[ \varphi(S) \subseteq  ((S^{-1}R)_{Q,\overline{\sigma}}[x_1^{\pm 1},\ldots,x_r^{\pm 1},x_{r+1},\ldots,x_n])^* \]

\item If $f:=\sum a_tx^t \in \mathcal{O}_{Q,\sigma}$ and $\varphi(f)=0$ using the unique representation of the elements in $\mathcal{O}_{Q,\overline{\sigma}}$ we conclude that $\frac{a_t}{1}=0$, for every $t \in \mathbb{Z}^r \times \mathbb{N}^{n-r}$ in the representation of $f$, since $R$ is a domain $a_t=0$, hence $f=0$. Conversely, if $s\cdot f=0$ for some $s \in S$, then $f=0$, and therefore, $\varphi(f)=0$. Moreover, $\varphi$ is injective.

\item Let $f:=\frac{a_1}{s_1}x^{t_1}+\cdots \frac{a_n}{s_n}x^{t_n} \in (S^{-1}R)_{Q,\overline{\sigma}}[x_1^{\pm 1},\ldots,x_r^{\pm 1},x_{r+1},\ldots,x_n]$. There are $s \in S$ and $b_1, \ldots,b_n \in R$ such that $\frac{b_i}{s}=\frac{a_i}{s_i}$, $i=1,\ldots, n$. Therefore
\begin{eqnarray}
f &=& \frac{b_1}{s}x^{t_1}+\cdots \frac{b_n}{s}x^{t_n}\nonumber \\
&=& (\frac{s}{1})^{-1}(\frac{b_1}{1}x^{t_1}+\cdots \frac{b_n}{1}x^{t_n})\nonumber \\
&=& \varphi(s)^{-1}\varphi(b_1x^{t_1}+\cdots b_nx^{t_n})\nonumber.
\end{eqnarray}
So, any element of $(S^{-1}R)_{Q,\overline{\sigma}}[x_1^{\pm 1},\ldots,x_r^{\pm 1},x_{r+1},\ldots,x_n]$ can be written in the form $\varphi(s)^{-1}\varphi(g)$ for some $s \in S$ and $g \in \mathcal{O}_{Q,\sigma}$.
\\

We conclude that $S^{-1}\mathcal{O}_{Q,\sigma}$ exist and \[S^{-1}\mathcal{O}_{Q,\sigma}\cong \mathcal{O}_{Q,\overline{\sigma}}:=(S^{-1}R)_{Q,\overline{\sigma}}[x_1^{\pm 1},\ldots,x_r^{\pm 1},x_{r+1},\ldots,x_n]\]
\end{enumerate}

\begin{lem}\label{lema4.4.4}
Let $\lambda \in End(\mathcal{O}_{Q,\sigma})$ be, then there is $\overline{\lambda} \in End(\mathcal{O}_{Q,\overline{\sigma}})$, i.e.  $\overline{\lambda}(\frac{a}{s})=\frac{a}{s}$ for every $\frac{a}{s} \in S^{-1}R$.
\end{lem}

\begin{proof}
Consider the homomorphism $\varphi \lambda: \mathcal{O}_{Q,\sigma} \longrightarrow \mathcal{O}_{Q,\overline{\sigma}}$, that satisfy $\varphi \lambda (S) \subseteq (\mathcal{O}_{Q,\overline{\sigma}})^*$. Then there is $\overline{\lambda} \in End(\mathcal{O}_{Q, \overline{\sigma}})$ such that $\varphi \lambda = \overline{\lambda} \varphi$.
\\

Now,
\begin{eqnarray}
\overline{\lambda}(\frac{a}{s})&=& \overline{\lambda}(\frac{1}{s}\frac{a}{1}) \nonumber \\
&=& \overline{\lambda}(\varphi(s))^{-1}\overline{\lambda}(\varphi(a)) \nonumber \\
&=& \varphi(\lambda(s))^{-1}\varphi(\lambda(a)) \nonumber \\
&=& \frac{a}{s},
\end{eqnarray}
thereby terminating the proof.
\end{proof}

\begin{thm}\label{theorem4.4.5}
Let $\mathcal{O}_{Q,\sigma}$ be a skew quantum polynomials ring over a Ore domain $R$, such that $q_{ij}$, $1\leq i < j \leq n$, are independent in $(S^{-1}R)^*/N$. If $\lambda \in End(\mathcal{O}_{Q, \sigma})$ and there are at least three indices $i,j,u$ with $\lambda(x_i),\lambda(x_j),\lambda(x_u) \neq 0$, then there are $\lambda_1,\ldots, \lambda_n \in R$ such that $\lambda(x_w)=\lambda_wx_w^{ \epsilon}$ for $w=1,\ldots, n$ and $\epsilon=\pm 1$.

\end{thm}

\begin{proof}
If $\lambda(x_i) \neq 0$, then $\overline{\lambda}(x_i)\neq 0$ (using the construction of $\overline{\lambda}$), therefore, there are at lest three indices $i,j,u$ such that $\overline{\lambda}(x_i),\overline{\lambda}(x_j),\overline{\lambda}(x_u)\neq 0$.
\\

According to the hypothesis, $\mathcal{O}_{Q,\overline{\sigma}}$ is a general skew quantum polynomial ring over $S^{-1}R$. So, there are $\lambda_1,\ldots, \lambda_n \in S^{-1}R$ such that $\overline{\lambda}(x_w)=\lambda_wx_w^\epsilon$ for $w=1,\ldots, n$ and $\epsilon = \pm 1$ (the arguments used in the proof of Theorem \ref{theorem2.1-2001}  are also valid to deduce this fact).
\\

Suppose that $\lambda(x_i)=\sum a_tx^t$, then $\sum \frac{a_t}{1}x^t=\varphi(\lambda(x_i))=\overline{\lambda}(\varphi(x_i))=\overline{\lambda}(x_i)=\lambda_ix_i^\epsilon$. Given the unique representation of the elements of $\mathcal{O}_{Q, \overline{\sigma}}$, $\frac{a_t}{1}=0$ if $t\neq (0,\ldots,0,\overbrace{\epsilon}^{i},0\ldots,0)=:i$ and $\lambda_i=\frac{a_i}{1}$, therefore $a_t=0$ if $t \neq i$, i.e, $\lambda(x_i)=a_ix_i^\epsilon$. So,  we concluded that 
\begin{eqnarray}
\lambda(x_w)=\lambda_wx_w^\epsilon,
\end{eqnarray}
for $\lambda_1,\ldots, \lambda_n \in R$ and $\epsilon = \pm 1$.
\end{proof}

\begin{cor}\label{corolary4.4.6}
If $\mathcal{O}_{Q, \sigma}$, in the situation of theorem \ref{theorem4.4.5}, and $\lambda \in Aut(\mathcal{O}_{Q,\sigma})$. Then, there are $\lambda_1, \ldots, \lambda_n \in R^*$, such that
\begin{eqnarray}
\lambda(x_w)=\lambda_wx_w^\epsilon,
\end{eqnarray}
for $w=1,\ldots, n$ and $\epsilon = \pm 1$.
\end{cor}

{\color{red}
\section{Examples}}
The following are examples in that the results of this paper cn be applied.
\\

\begin{enumerate}
\item Let $R[x_1;\sigma_1,\delta_1]\cdots [x_n; \sigma_n,\delta_n]$ be an iterated skew polynomial ring of bijective type, i.e., the following conditions hold:
\begin{enumerate}
\item for $1\leq i \leq n$, $\sigma_i$  is bijective;
\item for every $r \in R$ and $1 \leq i \leq  n$, $\sigma_i(r)$, $\delta_i(r)\in R$;
\item for $i < j$, $\sigma_j (x_i ) = cx_i + d$, with $ c, d \in R$ and $c$ has a left inverse;
\item for $i < j$, $\delta_j (xi ) \in R + Rx_1 + \cdots + Rx_n$ ,
\end{enumerate}
then, $R[x_1;\sigma_1,\delta_1]\cdots [x_n; \sigma_n,\delta_n]$ is a bijective skew $PBW$ extension.
In particular, quantum polynomial rings with $r=0$.
\item Diffusion algebras: a diffusion algebra $\mathcal{A}$ is generated by $\{ D_i,x_i : 1\leq i \leq n\}$ over $k$ with relations 
\begin{eqnarray}
x_ix_j=x_jx_i, 	\quad x_iD_j=D_jx_i 	\quad 1\leq i \leq n. \nonumber \\
c_{ij}D_iD_j-c_{ji}D_jD_i=x_jD_i-x_iD_j, \quad i <j, c_{ij},c_{ji} \in k^*.
\end{eqnarray}
Thus, $A \cong \sigma(k[x_1,\ldots,x_n])\langle D_1,\ldots, D_n \rangle$.

\item In the quantum algebras we have the following cases:

\begin{enumerate}
\item Multiplicative analogue of the Weyl algebra. The $k-$algebra $\mathcal{O}_n(\lambda_{ji})$ is generated by $x_1,\ldots,x_n$ subject to the relations
\begin{eqnarray}
x_jx_i=\lambda_{ji}x_ix_j, \quad 1\leq i < j \leq n,
\end{eqnarray}
where $\lambda_{ji} \in k^*$. We note that $\mathcal{O}_n(\lambda_{ji})\cong \sigma(k)\langle x_1,\ldots,x_n \rangle$.

\item $3-$dimensional skew polynomial algebra $\mathcal{A}$. It is given by the relations
\begin{eqnarray}
yz-\alpha zy=\lambda, \quad zx-\beta xz=\mu, \quad xy-\gamma yx=\nu, \nonumber
\end{eqnarray}
such that $\lambda, \mu, \nu \in k+kx+ky+kz$, and $\alpha,\beta, \gamma \in k^*$. Thus, $\mathcal{A}\cong \sigma(k)\langle x,y,z \rangle$.

\item The algebra of differential operators $D_q(S_q)$ on a quantum space $S_q$. Let $k$ be a commutative ring and le $q=[q_{ij}]$ be a matrix with entries in $k^*$ such that $q_{ii}=1=q_{ij}q_{ji}$ for all $1\leq i,j \leq n$, subject to the relations $x_ix_j=q_{ij}x_jx_i$. The algebra $S_q$ is regarded as the algebra of functions on a quantum space. The algebra $D_q(S_q)$ of $q-$differential operators on $S_q$ is defined by
\[ \partial_ix_j-q_{ij}x_j\partial_i=\delta_{ij} \quad \text{for all }i,j.\]
The relations between $\partial_i$, are given by
\[ \partial_i \partial_j=q_{ij}\partial_j\partial_i, \text{for all }i,j.\]
Therefore, $D_q(S_q)\cong \sigma(\sigma(k)\langle x_1,\ldots,x_n \rangle)\langle \partial_1,\ldots,\partial_n \rangle$.
\end{enumerate}
\end{enumerate}


\end{document}